\documentclass[a4paper,10pt]{amsart}
\usepackage[T1]{fontenc}
\usepackage{graphicx}
\usepackage{tikz,pgfplots}
\pgfplotsset{compat=1.6}
\usepackage{hyperref}
\usepackage{amssymb}
\usepackage{amsmath}
\usepackage{amsfonts}
\usepackage{amsthm}
\usepackage{amscd}
\usepackage{mathrsfs}
\usepackage{lineno}

\theoremstyle{plain}
\newtheorem{thm}{Theorem}[section]
\newtheorem{prop}[thm]{Proposition}

\newtheorem{lem}[thm]{Lemma}
\newtheorem{cor}[thm]{Corollary}
\newtheorem{rmk}[thm]{Remark}

\newtheorem{conj}[thm]{Conjecture}

\newtheorem*{theorem*}{Theorem}
\newtheorem*{conj*}{Conjecture}

\theoremstyle{definition}
\newtheorem{defin}[thm]{Definition}

\newcommand{\K}{\mathbb{K}}

\newcommand{\p}{\mathbb{P}}

\newcommand{\cal}{\mathcal}
\newcommand{\ov}{\overline}

\DeclareMathOperator{\ann}{Ann}
\DeclareMathOperator{\Hilb}{Hilb}

\DeclareMathOperator{\codim}{codim}

\title{On Minimal Gorenstein Hilbert functions}

\author{Lenin Bezerra}
\address{**Departamento de Matematica, Universidade Federal de Pernambuco, Recife PE (Brasil)}
\email{lenin.bezerra@ufpe.br}

\author{Rodrigo Gondim**}
\address{**Departamento de Matematica, Universidade Federal Rural de Pernambuco, Recife PE (Brasil)}
\email{rodrigo.gondim@ufrpe.br}

\author{Giovanna Ilardi*}
\address{*Dipartimento Matematica ed Applicazioni ``Renato Caccioppoli'', Universit\`a Degli Studi Di Napoli ``Federico II'', Via Cinthia - Complesso Universitario Di Monte S. Angelo 80126 - Napoli - Italia}
\email{giovanna.ilardi@unina.it}

\author{Giuseppe Zappal\`a***}
\address{***Dipartimento di Matematica e Informatica,  Universit\`a di Ca\-tania,
              	Viale A. Doria 6, 95125 Catania, Italy}
\email{zappalag@dmi.unict.it}

\begin{document}

\begin{abstract}

We conjecture that a class of Artinian Gorenstein Hilbert algebras called full Perazzo algebras always have minimal Hilbert function, fixing codimension and length. We prove the conjecture in length four and five, in low codimension. We also prove the conjecture for a particular subclass of algebras that occurs in every length and  certain codimensions. As a consequence of our methods we give a new proof of part of a known result about the asymptotic behavior of the minimum entry of a Gorenstein Hilbert function. 	 
   
\end{abstract}

\maketitle

\section*{Introduction}

Gorenstein algebras appear as cohomology rings in several categories. For instance, real orientable manifolds, projective varieties, Kahler manifolds, convex polytopes, matroids,  Coxeter groups and tropical varieties are examples of categories for which the ring of cohomology is an Artinian Gorenstein $\K$-algebra. The fundamental point is that these algebras can be characterized as algebras satisfying Poincar\'e duality, see \cite{MW}.

We deal with standard graded Artinian Gorenstein $\K$-algebras over a field of characteristic zero. A natural and classical problem consists in understanding their possible Hilbert function, sometimes also called Hilbert vector. When the codimension of the algebra is less than or equal to $3$, all possible Hilbert vectors were characterized in \cite{St}; in particular, they are unimodal, i.e. they never strictly increase after a strict decrease.
While it is known that non unimodal Gorenstein $h$-vectors exist in every codimension greater than or equal to $5$ (see \cite{BI, BO, BL}),
it is open whether non unimodal Gorenstein $h$-vectors of codimension $4$ exist.
For algebras with codimension $4$ having small initial degree the Hilbert vector is unimodal (see \cite{SS, MNZ3}).
\par

Consider the family $\cal{AG}_{\K}(r,d)$ of standard graded Artinian Gorenstein $\bf{\K}$-algebras of socle degree $d$ and codimension $r$.  
By Poncar\'e duality, the Hilbert function of $ A\in \cal{AG}_{\K}(r,d)$ is a symmetric vector $\Hilb(A)= (1,r,h_2, \ldots, h_{d-2}, r, 1)$, that is $h_k=h_{d-k}$.
There is a natural partial order in this family given by:
	\[(1,r,h_2,\dots, h_{d-2}, r,1) \preceq(1,r, \tilde{h}_2,\dots, \tilde{h}_{d-2}, r,1), \]  
  if $h_i \leq \tilde {h}_i$, for all $i \in \{2, \ldots, d-2\}$. The maximal Hilbert functions are associated to compressed algebras and completely described in \cite{IK}. In fact the Hilbert vector of a compressed Gorenstein algebra is a maximum in $\cal{AG}_{\K}(r,d)$.
  On the other hand, classifying minimal Hilbert functions is a hard problem. We do not know in general if there is a minimum.
  Moreover, given two comparable Gorenstein Hilbert functions, it is not true that any symmetric vector between them is Gorenstein. Some partial results in this direction were obtained in \cite{Za2} and called the interval conjecture.\par

The first example of a non-unimodal Gorenstein $h$-vector was given by Stanley (see \cite[Example 4.3]{St}). He showed that the $h$-vector $(1,13,12,13,1)$ is indeed a
Gorenstein $h$-vector. In \cite{MZa} the authors showed that Stanley's example is optimal, i.e. if we consider the $h-$vector $(1,12,11,12,1)$, it is not Gorenstein.
We say that a vector is totally non unimodal if
$$h_1>h_2>\ldots>h_k\ \text{for}\ k=\lfloor d/2 \rfloor. $$
 A totally non-unimodal Gorenstein Hilbert vector exists for every socle degree $d\geq 4$ when the codimension $r$ is large enough. It is related to a conjecture posed by Stanley and proved in \cite{MNZ, MNZ2} and also a consequence of our Proposition \ref{prop:hilb_perazzo}, see Corollary \ref{cor:totnonunimodal}. \par

From Macaulay-Matlis duality, every standard graded Artinian Gorenstein $\K$-algebra can be presented by a quotient of a ring of differential operators by a homogeneous ideal
that is the annihilator of a single form in the dual ring of polynomials. full Perazzo algebras are associated with full Perazzo polynomials, they are the family that we will study in detail.
Perazzo polynomials are related to Gordan and Noether theory of forms with vanishing Hessian (see \cite[Chapter 7]{Ru} and \cite{G}). In \cite{G} the author introduced the terminology Perazzo algebras to denote the Artinian Gorenstein algebra associated to a Perazzo polynomial. In \cite{FMM, AAdPFIMN} the authors study the Hilbert vector and the Lefschetz properties for Perazzo algebras in codimension $5$. In \cite{CGIZ} the authors study full Perazzo algebras focusing on socle degree $4$, showing that they have minimal Hilbert vector in some cases. In this paper we deal with codimension greater than $13$ and we are more interested in full Perazzo algebras. In the case of socle degree $4$ we recall the known results.  \par

We now describe the contents of the paper in more detail. In the first section we recall the basics on Macaulay-Matlis duality, see \ref{G=ANNF}. In the next subsection we recall the classical bounds for Hilbert functions given by Macaulay,
Gotzman and Green summarized in \ref{GGM}.\par

In the second section we recall the definition of full Perazzo algebras and we pose the full Perazzo Conjecture (see Conjecture \ref{conj:weak}). A full Perazzo polynomial of type $m$ and degree $d$ is a bigraded polynomial of bidegree $(1,d-1)$ given by $f=\sum x_jM_j$ where $\{M_j|j =1, \ldots, \binom{m+d-2}{d-1}\}$ is a basis for $\K[u_1, \ldots, u_m]_{(d-1)}$. The associated Artinian Gorenstein algebra is called full Perazzo algebra (see Section $2$ for more details).

 \begin{conj*} Let $H$ be the Hilbert vector of a full Perazzo algebra of type $m\geq 3$ and socle degree $d\geq 4$ and let $r=r(m,d)$ its codimension.
 Then $H$ is minimal in the family of Hilbert vectors of Artinian Gorenstein algebras of codimension $r$ and socle degree $d$, that is, if $\hat{H}$ is a comparable Artinian Gorenstein Hilbert vector such that $\hat{H} \preceq H$, then $\hat{H}=H$.
  \end{conj*}

In the third section we prove special cases of the Conjecture in socle degree $4$ and we try to fill the gaps in order to classify all possible Hilbert functions
up to codimension $25$ (see Theorem \ref{ventiq}, Corollary \ref{vintecinco} and Proposition \ref{twocases}).
In socle degree $5$ we prove the Conjecture for $m \in \{3,4,5,6,7,8,9,10\}$ (see Theorem \ref{socfive}) and a stronger version of the conjecture for $m = 3$ (see Corollary \ref{socfstrong}).\par

In the fourth section we prove our main result that the full Perazzo Conjecture is true for arbitrary socle degree $d \geq 4$ and type $m=3$.

\begin{theorem*} Every full Perazzo algebra with socle degree $d \geq 4$ of type $m=3$ has minimal Hilbert function.

\end{theorem*}

In the last section we give a new proof of part of a result originally proved in \cite{MNZ2}, concerning the asymptotic behavior of the minimum entry of a Gorenstein Hilbert function (see Theorem \ref{thm:MNZ}).  

\section{Preliminaries}

Most of the background material presented here can be found in \cite{CGIZ}.

\subsection{Macaulay-Matlis duality}

In this section we recall some basic results from Macaulay-Matlis duality for Artinian Gorenstein algebras over a field $\K$ of characteristic zero.
We recall that in characteristic zero we can use a differential version of Macaulay-Matlis duality.

Let $A=\K[X_1,\ldots,X_n]/I = \bigoplus_{i=0}^d A_i$ be a standard graded Artinian $\K-$algebra with $A_d \neq 0$.
The Hilbert function of $A$ can be described by the vector $\Hilb(A) = (1, h_1, \ldots, h_d)$ where $h_i=\dim A_i$.
We say that $A$ is Gorenstein if $\dim A_d =1$ and for every $i =1, \ldots, d$, the natural pairing given by multiplication $A_i \times A_{d-i} \to A_d \simeq \K $ is perfect.
There is an isomorphism $A_i^* \simeq A_{d-i}$.
In this context, $d$ is the socle degree of the algebra and assuming $I_1=0$, $n$ is the codimension of $A$.  \par

Let us regard the polynomial algebra $R=\K[x_1, \ldots, x_n]$ as a module over the algebra $Q=\K[X_1, . . . ,X_n]$ via the identification $X_i = \partial/\partial x_i.$ If $f\in R$ we set
 \[\ann_Q(f)=\{\alpha = p(X_1,\ldots,X_n)\in Q\mid \alpha(f):= p(\partial/\partial x_1,\ldots,\partial/\partial x_n)f=0\}.\]
 More generally, given any $Q$ submodule $M$ of $R$ we define the ideal of $Q$:
 \[\ann_Q(M)=\{\alpha\in Q\mid \alpha(f)=0 \ \text{for all} \ f \in M\}.\]
 On the other side we have the notion of inverse system. Given $I \subset Q$ be an ideal, we define the inverse system $I^{-1}$ which is a $Q$ submodule of $R$:
 \[I^{-1} = \{f \in R \mid \alpha(f) =0  \  \text{for all} \ \alpha \in I\}.\]

 By Macaulay-Matlis duality we have a bijection:

$$\begin{array}{ccc}
\{\text{Homogeneous ideals of} \ Q\} & \leftrightarrow & \{\text{Graded}\ Q\ \text{submodules} \ \text{of} \ R\}\\
\operatorname{Ann}_Q(M) & \leftarrow & M\\
I & \rightarrow & I^{-1}
  \end{array}
$$
 

From the Theory of Inverse Systems, we get the following characterization of standard
Artinian Gorenstein graded $\mathbb K$-algebras. A proof of this result can be found in \cite[Theorem 2.1]{MW}.

\begin{thm}{\bf (Double annihilator Theorem of Macaulay)} \label{G=ANNF} \\
Let $R = \K[x_1,\ldots,x_n]$ and let $Q = \K[X_1,\ldots, X_n]$ be the ring of differential operators.
Let $A= \bigoplus_{i=0}^dA_i = Q/I$ be an Artinian standard graded $\K$-algebra. Then
$A$ is Gorenstein if and only if there exists $f\in R_d$
such that $A\simeq Q/\operatorname{Ann}(f)$.
\end{thm}

In the sequel we always assume that $\operatorname{char}(\K)=0$, $A=Q/I$, $I=\ann_Q(f)$ and $I_1=0$.

We will deal with standard bigraded Artinian Gorenstein algebras $A=\bigoplus_{i=0}^d A_i,$ $A_d\ne 0$,
with $A_k=\bigoplus_{i=0}^k A_{(i,k-i)}$, $A_{(d_1,d_2)}\ne 0$ for some $d_1,d_2$ such that $d_1+d_2=d$, we call $(d_1,d_2)$ the socle bidegree of $A$. Since $A^*_k \simeq A_{d-k} $ and since duality is compatible with
direct sum, we get $A_{(i,j)}^* \simeq A_{(d_1-i,d_2-j)}$. \\
Let $R=\K[x_1,\ldots,x_n,u_1,\ldots,u_m]$ be the polynomial ring viewed as standard bigraded ring in the sets of variables $\{x_1,\ldots,x_n\}$ and $\{u_1,\ldots,u_m\}$
and let $Q=\K[X_1,\ldots,X_n,U_1,\ldots,U_m]$ be the associated ring of differential operators.

We want to stress that the bijection given by Macaulay-Matlis duality preserves bigrading, that is, there is a bijection:

$$\begin{array}{ccc}
\{\text{Bihomogeneous ideals of} \ Q\} & \leftrightarrow & \{\text{Bigraded}\ Q\ \text{submodules} \ \text{of}\ R\}\\
\operatorname{Ann}_Q(M) & \leftarrow & M\\
I & \rightarrow & I^{-1}
  \end{array}
$$

 If $f \in R_{(d_1,d_2)}$ is a bihomogeneous polynomial of total degree $d=d_1+d_2$,
then $I = \ann_Q(f) \subset Q$ is a bihomogeneous ideal and $A = Q/I$ is a standard bigraded Artinian Gorenstein algebra of socle bidegree $(d_1,d_2)$ and codimension $r=m+n$ if we assume, without loss of generality, that $I_1=0$.

\begin{rmk}\rm \label{defin:bigraded1}
 With the previous  notation, all bihomogeneous polynomials of bidegree $(1,d-1)$ can be written in the form
 $$f= x_1g_1+\dotsm+x_ng_n,$$
 where $g_i \in \K[u_1,\ldots,u_m]_{d-1}$. The associated algebra, $A = Q/\ann_Q(f)$, is bigraded, has socle bidegree $(1,d-1)$ and we assume that $I_1=0,$ so $\codim A = m+n$.
\end{rmk}

\subsection{Classical Bounds of Hilbert function}

We recall some classical bounds for the growth of the  Hilbert function of Artinian $\K-$algebras. The three main results are due to Macaulay, Gotzmann and Green; before stating them, we need to recall the following definition:
\begin{defin}
Let $k$ and $i$ be positive integers. The $i-$binomial expansion of $k$, denoted by $k_{(i)},$ is
\begin{equation}\label{eq:bin} k=k_{(i)}=\binom{k_i}{i}+\binom{k_{i-1}}{i-1}+\cdots+\binom{k_{j}}{j}\end{equation} where $k_i>k_{i-1}>\ldots>k_j\geq j\geq 1$.\end{defin}
An expansion of type (\ref{eq:bin}) always exists and is unique (see, e.g., \cite[Lemma 4.2.6]{BH}). Following \cite{BH}, we define for any integers $a$ and $b$, $$(k_{(i)})^b_a=\binom{k_i+b}{i+a}+\binom{k_{i-1}+b}{i-1+a}+\cdots+\binom{k_{j}+b}{j+a}$$
where we set $\binom{s}{c}=0$ whenever $s<c$ or $c<0$.

\begin{thm}\label{GGM}
Let $A=R/I$ be a standard graded $\K$-algebra, and $L\in A$ a general linear form (according to the Zariski topology). Denote by $h_d$ the degree $d$ entry of the Hilbert function of $A$ and by $h'_d$ the degree $d$ entry of the Hilbert function of $A/(L)$. Then:
\begin{description}
\item[(Macaulay)]
$$h_{d+1}\leq ((h_d)_{(d)})^{+1}_{+1}.$$
\item[(Gotzmann)] If $h_{d+1}=((h_d)_{(d)})^{+1}_{+1}$ and $I$ is generated in degrees $\leq d+1$, then $$h_{d+s}= ((h_d)_{(d)})^s_s\mbox{ for all } s\geq 1.$$
\item[(Green)] $$h'_d\leq  ((h_d)_{(d)})^{-1}_0.$$
\end{description}
\end{thm}
\begin{proof}
For \emph{Macaulay}, see \cite[Theorem 4.2.10]{BH}. For \emph{Gotzmann}, see \cite[Theorem 4.3.3]{BH} or \cite{Go}. For \emph{Green}, see \cite[Theorem 1]{Gr}.
\end{proof}

\begin{defin}
A sequence of non-negative integers $h=(1,h_1,h_2,\ldots,h_i,\ldots)$ is said to be an $\mathcal{O}$-sequence if it satisfies Macaulay's Theorem (\ref{GGM}) for all $i$.
\end{defin}

Recall that when A is Artinian and Gorenstein, then its Hilbert function is a finite, symmetric $\mathcal{O}$-sequence.

\section{Minimal Gorenstein Hilbert functions}

We recall the construction of full Perazzo algebras, introduced in \cite{CGIZ}.

 \begin{defin}
 Let $\K[x_1,\ldots,x_n,u_1,\ldots,u_m]$ be the polynomial ring in the $n$ variables $x_1,\ldots,x_n$ and in the $m$ variables $u_1,\ldots,u_m.$ A {\bfseries{Perazzo polynomial}}
 is a reduced bihomogeneous polynomial $f\in\K[x_1,\ldots,x_n,u_1,\ldots,u_m]_{(1,d-1)}$, of degree $d,$ of the form
 \begin{equation}\label{eqPerazzo}
 f=\displaystyle{\sum_{i=1}^nx_ig_i}
 \end{equation} with $g_i\in\K[u_1,\ldots,u_m]_{d-1}$, for $i=1,\ldots,n$, linearly independent and algebraically dependent polynomials in the variables $u_1,\ldots,u_m$.
 The associated algebra is called a {\bf Perazzo algebra}, it has codimension $m+n$ and socle degree $d$.  \end{defin}

Now we fix $m\geq 2$ and we consider the $m$ variables $u_1,\ldots,u_m$. For a multi-index $\alpha=(e_1, \ldots, e_m)$ with $e_{1}+\cdots+e_{m}=d-1$, let
 $$M_{\alpha}=u_{1}^{e_{1}}\cdots u_{m}^{e_{m}}\in Q_{d-1}$$  
   be a $\K$-linear basis for $Q_{d-1} $ and denote $\tau_m = \dim Q_{d-1} =\binom{m+d-2}{d-1}$.\par

  \begin{defin}\rm
  Let $f\in \K[x_1,\ldots,x_{\tau_m},u_1,\ldots,u_m]_{(1,d-1)}$ be a Perazzo polynomial  of degree $d$ of form:
\begin{equation}\label{eqcub}
	f=\displaystyle{\sum_{j=1}^{\tau_m} x_jM_j}.
\end{equation}
In this case $f$ is called {\bfseries{full Perazzo polynomial}} of type $m$ and degree $d$. The associated algebra is a {\bf full Perazzo algebra} of socle degree $d$ and codimension $m+ \tau_m$.    
  \end{defin}
 
 \begin{prop}\label{prop:hilb_perazzo}
 Let $A$ be a full Perazzo algebra of type $m\geq 2$ and socle degree $d$. Then for $k=0,\ldots, \lfloor\frac{d}{2}\rfloor$ \[h_k=\dim A_k=\binom{m+k-1}{k}+\binom{m+d-k-1}{d-k}. \]
 In particular, its Hilbert function is totally non-unimodal for $r>>0$.
 \end{prop}
 \begin{proof}
  Using the bigrading of $A$ and considering that the polynomial $f$ has degree $1$ in the  variables $x_1,\ldots,x_{\tau_m}$, fixed $k=0,\ldots,\lfloor\frac{d}{2}\rfloor$, we have the following decomposition:
  $$A_k=A_{(0,k)}\oplus A_{(1,k-1)}.$$
  \begin{enumerate}
  \item[(i)] It is clear that $A_{(0,k)}=Q_{(0,k)}$, hence $\dim A_{(0,k)}=\dim Q_{(0,k)}= \binom{m+k-1}{k}$.
  \item[(ii)] We have $A^*_{(1,k-1)}\simeq A_{(0,d-k)}$ and $A_{(0,d-k)}=Q_{(0,d-k)},$ hence $\dim A_{(1,k-1)}=\dim Q_{(0,d-k)}=\binom{m+d-k-1}{d-k}$.
  \end{enumerate}
 To verify that the Hilbert vector is asymptotically totally non unimodal it is enough to see that as a function of $m$, $h_k(m) \simeq \frac{1}{(d-k)!}m^{d-k}$ for $k \leq d/2$.  
 \end{proof}

 \begin{cor} \label{cor:totnonunimodal}
 	For every $d\geq 4$ there is a positive integer $r_0$ such that for all $r\geq r_0$ there is an Artinian Gorenstein algebra with socle degree $d$ and codimension $r$ having a totally non unimodal Hilbert vector.  
 \end{cor}

 \begin{proof} Let $m$ be large enough in order to guarantee that the Hilbert vector of the full Perazzo algebra $A=Q/\ann(f)$, of type $m $ and socle degree $d$ has a totally non unimodal Hilbert vector. For every $r > m+ \binom{m+d-2}{d-1}$, let $s=r - [m+ \binom{m+d-2}{d-1}]$ and consider the algebra
 $A'=Q'/\ann(g)$ where $Q'= Q[Y_1, \ldots, Y_s]$ and $g = f + \sum_{i=1}^s Y_i^d$. It is easy to see that the Hilbert vector of $A'$ is
 given by $h'_k = h_k + s$ for $k\neq 0,d$, therefore, it is totally non-unimodal and the result follows.
 \end{proof}

  Let $d \geq 4$, $r\geq 3$. Consider the family $\cal{AG}(r,d)$ of standard graded artinian Gorenstein $\bf{\K}$-algebras of socle degree $d$ and codimension $r$.
  In this section we will consider $\K$, a fixed field of characteristic $0$.
  We know that the Hilbert function of $ A\in \cal{AG}(r,d)$ is a symmetric vector $\Hilb(A)= (1,r,h_2, \ldots, h_{d-2}, r, 1)$, with $h_i=h_{d-i}$ by Poincar\'e duality.
 
  Consider the family of length $d$ symmetric vectors of type
  $(1,r,h_2, \ldots, h_{d-2}, r, 1)$, where $h_i=h_{d-i}$. There is a natural partial order in this family
 
  $$(1,r,h_2,\dots, h_{d-2}, r,1) \preceq (1,r, \tilde{h}_2,\dots, \tilde{h}_{d-2}, r,1). $$
 
  If $h_i \leq \tilde {h}_i$, for all $i \in \{2, \ldots, d-2\}$.
  This order can be restricted to $\cal{AG}(r,d)$ which becomes a {\it poset}.

 \begin{defin} Let $r,d$ be fixed positive integers and let $H$ be a length $d+1$ symmetric vector
  $(1,r,h_2, \ldots, h_{d-2}, r, 1)$.
  We say that $H$ is a  {\bfseries{minimal Artinian Gorenstein Hilbert function}} of socle degree $d$ and codimension $r$ if there is an Artinian Gorenstein algebra such that $\Hilb(A)=H$ and $H$ is minimal
  in $\cal{AG}(r,d)$ with respect to $\preceq$. To be precise, if $\hat{H}$ is a comparable Artinian Gorenstein Hilbert vector such that $\hat{H}\preceq H$, then $\hat{H}=H$.  
  \end{defin}
 
  We now present the full Perazzo Conjecture.
 
  \begin{conj}\label{conj:weak} Let $H$ be the Hilbert vector of a full Perazzo algebra of type $m$ and socle degree $d$. Then $H$ is minimal in $\cal{AG}(r,d)$.
 
  \end{conj}

 \section{Minimal Gorenstein Hilbert functions in low socle degree}

In this section we study Gorenstein Hilbert functions of algebras with socle degree $4$ and $5.$ Part of the results in socle degree $4$ can be found in \cite{CGIZ}.

\subsection{Minimal Gorenstein Hilbert functions in socle degree 4}

In socle degree $4,$ a Gorenstein sequence is of the form
   $$(1,r,h,r,1).$$
Let $\mu(r)$ be the integer such that $(1,r,\mu(r),r,1)$ is a Gorenstein sequence, but $(1,r,\mu(r)-1,r,1)$ is not a Gorenstein sequence. Then $\mu(r) \le h \le \binom{r+1}{2}.$
\par
It is well known that $(1,r,h,r,1)$ is a Gorenstein sequence if and only if $\mu(r) \le h \le \binom{r+1}{2}$ (see \cite{Za2}).  We set $\delta(r)=r-\mu(r).$ This function was introduced in \cite{MNZ} and also studied in \cite{CGIZ}. The function $\delta(r)$ is not decreasing, so $\delta(r) \le \delta(r+1),$ for every $r$ (see \cite{MNZ}, Proposition 8).
\par
By Remark 5.4 in \cite{CGIZ}, if $\delta(r-1) < \delta(r)$ then $\delta(r)=\delta(r-1)+1.$

\begin{defin}
We say that the Gorenstein sequence $(1,r,\mu(r),r,1)$ is {\em minimal}.
Moreover we say that the Gorenstein sequence $(1,r,\mu(r),r,1)$ is {\em strongly minimal} if $\delta(r-1) < \delta(r).$
\end{defin}

By Remark 5.4 in \cite{CGIZ}, if $(1,r,\mu(r),r,1)$ is strongly minimal, then $\delta(r)=\delta(r-1)+1.$
\par
The minimal $r$ such that $(1,r,\mu(r),r,1)$ is not unimodal is $r=13$ (\cite{MZa}). So $\delta(r)=0$ for $r \le 12.$

\begin{prop}\label{delta1}
$\delta(r)=1$ iff $13 \le r \le 19.$
\par
Consequently the sequence $(1,13,12,13,1)$ is strongly minimal.
\end{prop}
\begin{proof}
The sequence $(1,13,12,13,1)$ is a Gorenstein sequence. This was originally proved by Stanley in \cite{St}.
This sequence is also the Hilbert Function of the full Perazzo algebra with $m=3.$ In \cite{MZa}, Proposition 3.1, it was proved that $(1,12,11,12,1)$ is not a Gorenstein sequence.
Consequently $\delta(12)=0,$ therefore $\delta(r)=0$ for every $r \le 12.$ In \cite{AMS}, Theorem 4.1, was shown that $(1,19,17,19,1)$ is not a Gorenstein sequence,
so $\delta(r)=1$ for $13 \le r \le 19.$ In \cite{MZa}, Remark 3.5, was observed that $(1,20,18,20,1)$ is a Gorenstein sequence, so for $r \ge 20$ we have that $\delta(r)\ge 2.$
\end{proof}

\begin{cor}
$\delta(20)=2.$ Consequently the sequence $(1,20,18,20,1)$ is strongly minimal.
\end{cor}
\begin{proof}
In \cite{MZa}, Remark 3.5, it was observed that $(1,20,18,20,1)$ is a Gorenstein sequence.
So, by Remark 5.4 in \cite{CGIZ}, $\delta(20)=2.$
\end{proof}

\begin{prop} \label{perq}
Let $m \ge 3.$ We have that
  $$\delta\left(m+\binom{m+2}{3}\right)\ge \binom{m}{3}.$$ 
\end{prop}
\begin{proof}
For $r=m+\binom{m+2}{3}$ there exists the full Perazzo Algebra. It realizes the Gorenstein sequence
   $$(1,m+\binom{m+2}{3},m(m+1),m+\binom{m+2}{3},1).$$
So $\delta\left(m+\binom{m+2}{3}\right) \ge \binom{m+2}{3}+m-m(m+1)=\binom{m+2}{3}-m^2=\binom{m}{3}.$
\end{proof}

\begin{lem} \label{gors}
Let $(1,r,h,r,1)$ be a Gorenstein sequence. Let $u=r-h,$ with $u \ge 0$. Then
  $$\big( \big( (r_{(3)})^{-1}_{0}-u \big)_{(2)} \big)^{1}_{1} \ge (r_{(3)})^{-1}_{0}.$$
\end{lem}
\begin{proof}
Let $A$ be a Gorenstein algebra with Hilbert function $(1,r,h,r,1)$ and let $L$ be a general linear form.
Using the same argument as in Proposition 3.1 in \cite{MZa}, we get that the Hilbert function of $A/(L)$ is of the type
  $$(1,r-1,s-u,s).$$
By the theorems of Green and of Macaulay we have $s \le (r_{(3)})^{-1}_{0}$ and  
   $$\big( (s-u)_{(2)} \big)^{1}_{1} \ge s.$$
Consequently
   $$\big( (s+t-u)_{(2)} \big)^{1}_{1} \ge s+t,\,\text{ for every }t \ge 0;$$
In particular, for $t=(r_{(3)})^{-1}_{0}-s$ we are done.
\end{proof}

\begin{thm} \label{ventiq}
$\delta(24)=4$ and $\delta(40)=10.$
\end{thm}
\begin{proof}
By Proposition \ref{perq}, $\delta(24) \ge \binom{4}{3}=4.$ We have to prove that $(1,24,19,24,1)$ is not a Gorenstein sequence.
Indeed $24_{(3)}=\binom{6}{3}+\binom{3}{2}+\binom{1}{1},$ so $(24_{(3)})^{-1}_{0}=11.$ Since $u=5$ we have that
  $$\big((11-5)_{(2)}\big)^{1}_{1}=10 < 11.$$
By Lemma \ref{gors}, $(1,24,19,24,1)$ is not a Gorenstein sequence.
\par
By Proposition \ref{perq}, $\delta(40) \ge \binom{5}{3}=10.$ We have to prove that $(1,40,29,40,1)$ is not a Gorenstein sequence.
Indeed $40_{(3)}=\binom{7}{3}+\binom{3}{2}+\binom{2}{1},$ so $(40_{(3)})^{-1}_{0}=22.$ Since $u=11$ we have that
  $$\big((22-11)_{(2)}\big)^{1}_{1}=21 < 22.$$
By Lemma \ref{gors}, $(1,40,29,40,1)$ is not a Gorenstein sequence.    
\end{proof}

\begin{cor}\label{vintecinco}
$\delta(25)=4.$
\end{cor}
\begin{proof}
By Theorem \ref{ventiq} and by Theorem 2.5 in \cite{MZa}, $(1,25,21,25,1)$ is a Gorenstein sequence, so $\delta(25) \ge 4.$
We have to prove that $(1,25,20,25,1)$ is not a Gorenstein sequence.
Indeed $25_{(3)}=\binom{6}{3}+\binom{3}{2}+\binom{2}{1},$ so $(25_{(3)})^{-1}_{0}=12.$ Since $u=5$ we have that
  $$\big((12-5)_{(2)}\big)^{1}_{1}=11 < 12.$$
By Lemma \ref{gors}, $(1,25,20,25,1)$ is not a Gorenstein sequence.    
\end{proof}

\begin{prop} \label{twocases}
$2\le \delta(21) \le \delta(22)\le \delta(23) \le 4.$
\end{prop}
\begin{proof}
This follows trivially by the fact that $\delta(20)=2$ and $\delta(24)=4.$
\end{proof}

\begin{prop}
$20\le \delta(62)\le 21.$
\end{prop}
\begin{proof}
By Proposition \ref{perq}, for $m=6,$ we get that $(1,62,42,62,1)$ is a Gorenstein sequence.
\par
On the other hand, $(1,62,40,62,1)$ is not a Gorenstein sequence by Lemma \ref{gors}.
Indeed $$(62_{(3)})^{-1}_{0}=\binom{7}{3}+\binom{3}{2}=38$$
and
  $$\big( (38-22)_{(2)} \big)^{1}_{1}=36 < 38.$$
\end{proof}

\begin{prop}
$\delta(26) = 4 = \delta(27).$
\end{prop}
\begin{proof}
 For $r=26$, we have to prove that $(1,26,21,26,1)$ is not a Gorenstein sequence. Indeed, let $A=R/I$ be a Gorenstein algebra with Hilbert function $(1,26,21,26,1)$, $L$ be a general linear form. We set $J = (I_{\le 3})$, $\Bar{J} = (J,L)/(L)$ and $S=R/(L)$. By the theorems of Green and of Macaulay and repeating the above method, $R/\Bar{J}$ has Hilbert function $(1,25,8,13)$. As $R/\Bar{J}$ has maximal growth from degree $2$ to degree $3$ and $\Bar{J}$ has no new generators in degree $4$, by Gotzmann's theorem we get $h_{R/\Bar{J}}(t) = \binom{t+2}{2} + t.$ Therefore, $\Bar{J}$ is the saturated ideal, in all degrees $\ge 2$ of the union of a plane and a line in $\mathbb{P}^{24}$. It follows that, up to saturation, $J$ is the ideal of a scheme $T$ given by the union in $\mathbb{P}^{25}$ of a $3-$dimensional linear variety, a plane and $m$ points (possibly embedded). Hence, $50 \le h_{R/\Bar{J}}(4) \le 45$, which is absurd.

 Now, for $r=27$, following the same argument as above, we prove that the sequence $(1,27,22,27,1)$ is not Gorenstein. In this case we conclude $50 \le h_{R/\Bar{J}}(4) \le 46.$  
 
\end{proof}

\subsection{Minimal Gorenstein Hilbert functions in socle degree 5}

In socle degree $5,$ a Gorenstein sequence is of the form
   $$(1,r,h,h,r,1).$$
Let $\mu(r)$ be the integer such that $(1,r,\mu(r),\mu(r),r,1)$ is a Gorenstein sequence, but $(1,r,\mu(r)-1,\mu(r)-1,r,1)$ is not a Gorenstein sequence.
Then $\mu(r) \le h \le \binom{r+1}{2}.$
\par
It is well known that $(1,r,h,h,r,1)$ is a Gorenstein sequence iff $\mu(r) \le h \le \binom{r+1}{2}$ (see \cite{Za2}). We set $\delta(r)=r-\mu(r).$

\begin{defin}
We say that the Gorenstein sequence $(1,r,\mu(r),\mu(r),r,1)$ is {\em minimal}.
Moreover we say that the Gorenstein sequence $(1,r,\mu(r),\mu(r),r,1)$ is {\em strongly minimal} if $\delta(r-1) < \delta(r).$
\end{defin}

\begin{prop} \label{delf}
In socle degree $5$ we have
\par
 $\delta(r)=0$  iff  $r\le 16.$
\par
 $\delta(r)=1$ iff  $r=17.$
\par
For $18 \le r \le 25,$ $\delta(r)=2$
\end{prop}
\begin{proof}
For $18 \le r \le 25,$ see Theorem 3.4 and Remark 3.5 in \cite{MZa}.
\end{proof}

\begin{lem} \label{gorf}
Let $(1,r,h,h,r,1)$ be a Gorenstein sequence. Let $u=r-h.$ Then
  $$\big( \big( (r_{(4)})^{-1}_{0}-u \big)_{(2)} \big)^{2}_{2} \ge (r_{(4)})^{-1}_{0}.$$
\end{lem}
\begin{proof}
Analogous to Lemma \ref{gors}.
\end{proof}

\begin{prop} \label{perqf}
Let $m \ge 3.$ We have that
  $$\delta\left(m+\binom{m+3}{4}\right) \ge \frac{m+5}{4}\binom{m}{3}.$$ 
\end{prop}
\begin{proof}
For $r=m+\binom{m+3}{4}$ there exists the full Perazzo Algebra. It realizes the Gorenstein sequence
   $$\left(1,m+\binom{m+3}{4},\binom{m+1}{2} + \binom{m+2}{3},\binom{m+1}{2} + \binom{m+2}{3},m+\binom{m+3}{4},1\right).$$
So $$\delta\left(m+\binom{m+3}{4}\right) \ge \binom{m+3}{4}+m-\binom{m+1}{2} - \binom{m+2}{3}=\frac{m+5}{4}\binom{m}{3}.$$
\end{proof}

\begin{thm} \label{socfive}
For $m \in \{3,4,5,6,7,8,9,10\},$ we have $\delta\left(m+\binom{m+3}{4}\right) = \frac{m+5}{4}\binom{m}{3}.$ That is, the full Perazzo conjecture is true in these cases.
\end{thm}

\begin{proof}
By Proposition \ref{perqf}, we have to prove that for $m \in \{3,4,5,6,7,8,9,10\},$ the sequence
  {\footnotesize $$\left(1,m+\binom{m+3}{4},\binom{m+1}{2} + \binom{m+2}{3}-1,\binom{m+1}{2} + \binom{m+2}{3}-1,m+\binom{m+3}{4},1\right)$$}
is not a Gorenstein sequence.
\par
The case $m=3$  will be dealt with in general in the next section. We can assume $m\geq 4$.
\par
For $m=4$ we have to prove that the sequence $(1,39,29,29,39,1)$ is not a Gorenstein sequence. Indeed, using Lemma \ref{gorf}, we have:
  $$(39_{(4)})^{-1}_{0}=16;\,\,u=10;\,\,16-10=6;\,\,(6_{(2)})^{2}_{2}=15 < 16.$$
    
For $m=5$ we have to prove that the sequence $(1,75,49,49,75,1)$ is not a Gorenstein sequence. Indeed, by Lemma \ref{gorf}, we have:
  $$(75_{(4)})^{-1}_{0}=36;\,\,u=26;\,\,36-26=10;\,\,(10_{(2)})^{2}_{2}=35 < 36.$$
    
For $m=6$ we have to prove that the sequence $(1,132,76,76,132,1)$ is not a Gorenstein sequence. Indeed, using Lemma \ref{gorf}, we have:
  $$(132_{(4)})^{-1}_{0}=71;\,\,u=56;\,\,71-56=15;\,\,(15_{(2)})^{2}_{2}=70 < 71.$$
    
For $m=7$ we have to prove that the sequence $(1,217,111,111,217,1)$ is not a Gorenstein sequence. Indeed, by Lemma \ref{gorf}, we have:
  $$(217_{(4)})^{-1}_{0}=128;\,\,u=106;\,\,128-106=22;\,\,(22_{(2)})^{2}_{2}=127 < 128.$$
    
For $m=8$ we have to prove that the sequence $(1,338,155,155,338,1)$ is not a Gorenstein sequence. Indeed, using Lemma \ref{gorf}, we have:
  $$(338_{(4)})^{-1}_{0}=212;\,\,u=183;\,\,212-183=29;\,\,(29_{(2)})^{2}_{2}=211 < 212.$$

For $m=9$ we have to prove that the sequence $(1,504,209,209,504,1)$ is not a Gorenstein sequence. Indeed, using symmetry, Green's theorem and Macaulay's theorem, the following diagram represents the Hilbert functions of $R/I$, $R/(I:L)$ and $R/(I,L)$
$$\begin{array}{ccccccccc}
1 & 504 & 209 & 209 & 504 & 1\\

 & 1 & 171 & 54 & 171 & 1\\
\hline
1 & 503 & 38 & 155 & 333 & &   
\end{array}$$
By Lemma \ref{gors} the middle line is not a Gorenstein sequence.

For $m=10$ we have to prove that the sequence $(1,725,274,274,725,1)$ is not a Gorenstein sequence. Indeed, using symmetry, Green's theorem and Macaulay's theorem, the following diagram represents the Hilbert functions of $R/I$, $R/(I:L)$ and $R/(I,L)$
$$\begin{array}{ccccccccc}
1 & 725 & 274 & 274 & 725 & 1\\

 & 1 & 226 & 65 & 226 & 1\\
\hline
1 & 724 & 48 & 209 & 499 & &   
\end{array}$$
By Lemma \ref{gors} the middle line is not a Gorenstein sequence.
\end{proof}

\begin{cor} \label{socfstrong}
The Gorenstein vector \[\left(1,18,16,16,18,1\right)\]
is strongly minimal.
\end{cor}
\begin{proof}
By the previous theorem we know that it is a minimal Gorenstein Hilbert vector.
We have to prove that $(1,17,15,15,17,1)$ is not a Gorenstein sequence.
Indeed, by Proposition \ref{delf}, $\delta(17)=1.$
\end{proof}



\section{A family of minimal Gorenstein Hilbert functions}

Consider the  family of  full Perazzo algebras of type $m=3$ and socle degree $d\geq 4$. Its Hilbert function is given by
$h_k= \binom{k+2}{k} + \binom{2+2q-k}{2q-k}$, for $k\leq \lfloor d/2 \rfloor$ and by symmetry we get $h_{d-k}= h_k$.

\begin{lem}\label{tec} Let $k \leq \lfloor d/2 \rfloor$. Then we have:
$$\left(\left(\binom{k+1}{2}\right)_{(d-k)}\right)_{0}^{-1} \leq k-2.$$

\end{lem}

\begin{proof} First of all, consider $d> 2k-1$. In this case, $ d-k+1 >k$, i.e.
$(d-k+1)+ (d-k) + \ldots + (d-2k+3) > k+ (k-1)+ \ldots +2+1= k(k+1)/2= \binom{k+1}{2}$.

We have:
$(\binom{k+1}{2})_{(d-k)}<  \binom{d-k+1}{d-k}+ \binom{d-k}{d-k-1} + \ldots \binom{d-2k+3}{d-2k+2}$.
Then

 \[\left(\left(\binom{k+1}{2}\right)_{(d-k)}\right)^{-1}_{0} \leq k-2.\]

Now there are only two other cases to consider:
1) $d=2k-1$, 2) $d=2k$.

They are similar, we will do the details for $d=2k$. In this case, we have:

$$\left(\left(\binom{k+1}{2}\right)_{(k)}\right)^{-1}_{0} \leq k-2.$$

Indeed we know that the $k$-binomial expansion of $\binom{k+1}{2}$ has two blocks

$(\binom{k+1}{2})_{(k)} = [\binom{k+1}{k} + \binom{k}{k-1}+ \ldots+ \binom{j+1}{j} ] + [\binom{j-1}{j-1} + \ldots + \binom{i}{i}]$.
The first block consists of binomials of type $\binom{s+1}{s}$ and the second one of type $\binom{s}{s}$.

Therefore:

$$\left(\left(\binom{k+1}{2}\right)_{(k)}\right)^{-1}_{0}= k-j+1.$$

If $k-j+1 >k-2$, then $j\leq 2$, but the cases $j=1$ and $j=2$ are not possible.
In fact, suppose, $j=2$, since:
\[\binom{k+1}{k} + \binom{k}{k-1}+ \ldots+ \binom{3}{2}= (k+4)(k-1)/2 \ge \binom{k+1}{2}.\]
It is absurd for $k>2$. The case $j=1$ is analogous, the result follows.

\end{proof}
\begin{thm} \label{thm:main} Every full Perazzo algebra with socle degree $d \geq 4$ of type $m=3$ has minimal Hilbert function.

\end{thm}
\begin{proof}

We want to show that the Hilbert vector of the full Perazzo algebra $H=(1, h_1,h_2,h_3, \ldots,h_{d-1},h_d=1)$  with $h_k=\binom{m+k-1}{k}+\binom{m+d-k-1}{d-k}$
is a minimal Gorenstein Hilbert vector. Let $$\hat{H}=(1, \hat{h}_1,\hat{h}_2,\hat{h}_3, \ldots,\hat{h}_{d-1},1)$$ be a comparable Artinian Gorenstein Hilbert vector $\hat{H} \preceq H$ of length $d+1$ and $\hat{h}_1=h_1$. We will proceed in steps to show that $\hat{H} = H$. Consider, on the contrary, one of the following situations:

\begin{enumerate}
 \item For some  $k \in \{2, \ldots, \lfloor d/2 \rfloor -1  \}$, $\hat{h}_k<h_k$;
 
 \item For $d=2q$, suppose that $\hat{h}_t=h_t$ for all $t<q$ and $\hat{h}_q<h_q$;
 
 \item For $d=2q+1$, suppose that $\hat{h}_t=h_t$ for all $t<q$ and $\hat{h}_q<h_q$.
\end{enumerate}

We will show that all of these situations give rise to a contradiction.

(1). Let $A=Q/I$ with $I=\ann(f)$ be a standard graded Artinian Gorenstein $\K$-algebra such that $H_A=\hat{H}$ with
$\hat{h}_k =\dim A_k <h_k=\binom{m+k-1}{k}+\binom{m+d-k-1}{d-k}$ for some  $k \in \{2, \ldots, \lfloor d/2 \rfloor -1  \}$. Suppose that $k$ is minimal satisfying this property, that is, for $t<k$ we get $\hat{h}_t = h_t$, by the comparability hypothesis.
Let $L \in A_1$ be a generic linear form and let $S=Q/(L)$. We get the following exact sequence: $$\begin{CD}0@>>>Q/(I:L)(-1)@>>>Q/I@>>>S/\overline{I}@>>> 0\end{CD}$$ with
$\overline{I}=\dfrac{(I,L)}{L}$ and $(I:L)= Ann(f)$ and $f'=L(f)$ denoting the derivative of $f$ with respect to $L\in Q$. Therefore $Q/(I:L)$ is also Gorenstein. We get the following diagram:
$$\begin{array}{ccccccccc}
1 & \hat{h}_1 & \ldots &  \hat{h}_{k}& \ldots & \hat{h}_{d-k} & \hat{h}_{d-k+1} & \ldots & 1\\

 & 1 & \ldots & a_{k-1} & \ldots & & a_{k-1}  & \ldots & 1\\
\hline
1 & \hat{h}_{1}-1 & \ldots &  h_{k}' & \ldots & & h_{d-k+1}'&    
\end{array}$$

We have $\hat{h}_{d-k+1}= \hat{h}_{k-1} = h_{k-1} = h_{d-k+1} = \binom{k+1}{k-1} + \binom{d-k+3}{d-k+1}$. The $(d-k+1)$-binomial decomposition of $h_{d-k+1}$ is $(h_{d-k+1})_{(d-k+1)}=\binom{d-k+3}{d-k+1} + (\binom{k+1}{k-1})_{(d-k)}$. By Green's theorem we have
\begin{center}
$h_{d-k+1}' \leq ((\hat{h}_{d-k+1})_{(d-k+1)})_0^{-1} = ((h_{d-k+1})_{(d-k+1)})_0^{-1} =  \binom{d-k+2}{d-k+1} + \left(\left(\binom{k+1}{k-1}\right)_{(d-k)}\right)^{-1}_{0}.$
\end{center}
By Lemma \ref{tec}, we have
\begin{center}
$ h_{d-k+1}' \leq d-k+2 +k-2= d$.
\end{center}

We consider only the case $h'_{d-k+1}=d$, the other cases are similar.

We have  $a_{k-1}=\hat{h}_{d-k+1}-d$, $h'_k=\hat{h}_{k}-(\hat{h}_{d-k+1}-d)$. Since $\hat{h}_k \leq h_k-1$ we have $h_k'\le h_{k} -h_{d-k+1}+d-1 $.

We recall that
\begin{center}
$h_k=\binom{k+2}{2} + \binom{d-k+2}{2}$, $h_{d-k+1}= \binom{d-k+3}{2} +\binom{k+1}{2}$.    
\end{center}
Therefore

\begin{eqnarray*}
 h_{k}-h_{d-k+1} & = & \left[\binom{k+2}{2}- \binom{k+1}{2}\right]- \left[\binom{d-k+3}{2}- \binom{d-k+2}{2}\right] \\
& = & (k+1)-(d-k+2) \\
& = & 2k-d-1
\end{eqnarray*}

We obtain
$h'_k \le 2k-2$.
Thence $h'_k \le 2k-2= k+1+k-3$ which implies that $(h'_k)_k \leq \binom{k+1}{k} + \binom{k-1}{k-1} + \cdots + \binom{3}{3}$.

By Macaulay's theorem  applied $d-2k+1$ times  we have
\begin{center}
$h'_{d-k+1}\leq ((h'_{k})_{k})^{d-2k+1}_{d-2k+1} \leq \binom{k+1+d-2k+1}{k+d-2k+1}+k-3= k+1+d-2k+1+k-3$
\end{center}
 therefore $d\leq d-1$,a contradiction.\par

(2). Case $d=2q$ is even. Suppose that $\hat{h}_t=h_t$ for all $t<q$ and $\hat{h}_q<h_q$.
Let $L\in Q$ be a generic linear form and $S=Q/(L)$. We have the following exact sequence: $$\begin{CD}0@>>>Q/(I:L)(-1)@>>>Q/I@>>>S/\overline{I}@>>> 0\end{CD}$$
where $\overline{I}=\dfrac{(I,L)}{L}$ and $(I:L)= Ann(f)$, $f'=L(f)$, that is, $Q/(I:L)$ is also Gorenstein. In the middle we  get the following diagram:

$$\begin{array}{cccccccc}
1 & h_1 & \ldots & h_{q-1} & \hat{h}_{q}& h_{q+1} &  \ldots & 1\\
\\
& 1 & \ldots & a_{q-2} & a_{q-1} & a_{q-1} & \ldots & 1\\
\hline\\
1 & h_{1}-1 & \ldots & & h_{q}' & h_{q+1}'  
\end{array}$$

Since $(h_{q+1})_{(q+1)}=(\binom{q+3}{q+1} + (\binom{q+1}{2})_q$, from Green's theorem
\begin{center}
$h_{q+1}' \leq ((h_{q+1})_{q+1})^{-1}_{0} = \binom{q+2}{q+1} + ((\binom{q+1}{2})_{q})^{-1}_{0}$.
\end{center}
By Lemma \ref{tec} we have
\begin{center}
$((h_{q+1})_{q+1})^{-1}_{0} = \binom{q+2}{q+1} + ((\binom{q+1}{2})_{q})^{-1}_{0} \leq q+2+q-2=2q$.
\end{center}

We study the case $h_{q+1}'=2q$, the other cases are similar.

We have $a_{q-1}'= h_{q+1}-2q$, $h'_{q}= \hat{h}_{q}-a_{q-1} \le h_{q}-h_{q+1} +2q-1$, then
$h'_{q} \leq 2q-2= (q+1)+(q-3)$.

Therefore
\begin{center}
$(h'_{q})_{q} \leq \binom{q+1}{q} + \binom{q-1}{q-1}+ \binom{q-2}{q-2}+ \cdots + \binom{3}{3}$,
\end{center}
with $\binom{q-1}{q-1}+ \binom{q-2}{q-2}+ \cdots + \binom{3}{3}$ being counted $q-3$ times.

From Macaulay's theorem we have
$h'_{q+1} \leq ((h'_{q})_{(q)})^{+1}_{+1}$, hence

\begin{eqnarray*}
 2q & \leq & \binom{q+2}{q+1} + \binom{q}{q} + \binom{q-1}{q-1} + \cdots + \binom{4}{4} \\
& \leq & q+2+q-3=2q-1.
\end{eqnarray*}
It is a contradiction. \par

(3). If $d= 2q+1$ is odd. Suppose that $\hat{h}_t=h_t$ for all $t<q$ and $\hat{h}_q<h_q$. By the same argument:

$$\begin{array}{cccccccc}
1 & h_1 & \ldots & \hat{h}_{q} & \hat{h}_{q+1}& h_{q+2} &  \ldots & 1\\
\\
& 1 & \ldots & a_{q-1} & a_{q} & a_{q-1} &\ldots & 1\\
\hline\\
1 & h_{1}-1 & \ldots& h_{q}' & h_{q+1}'   & h_{q+2}'
\end{array}$$

Since $h_{q+2}= \binom{q+4}{q+2} + \binom{q+1}{2}$ and $(h_{q+2})_{q+2}= \binom{q+4}{q+2} + (\binom{q+1}{2})_{q+1} $, by Green's theorem
\begin{center}
$h_{q+2}'\leq ((h_{q+2})_{(q+2)})^{-1}_{0}= \binom{q+3}{q+2}+ (\binom{q+1}{2})_{(q+1)})^{-1}_{0} \leq q+3+q-2= 2q+1.$
\end{center}


We consider only the case $h'_{q+2}= 2q+1$. We have $a_{q-1}= h_{q+2}- (2q+1)$, $h'_{q}= \hat{h}_{q}-a_{q-1}$. Then $h'_{q} \le h_{q} -1- a_{q-1}$,  $h'_{q} \le h_{q} -1-( h_{q+2}- (2q+1))$, thence $h_q' \le 2q-2$.
We have
\begin{center}
$h_{q}- h_{q+2} = \binom{q+2}{2}- \binom{q+1}{2}+ \binom{q+3}{2}- \binom{q+4}{2}= -2$.
\end{center}

Therefore
\begin{eqnarray*}
h'_{q}& \leq & (q+1) +(q-3) \\
& \leq & \binom{q+1}{q} + \binom{q-1}{q-1} + \cdots + \binom{3}{3}
\end{eqnarray*}
where the terms $\binom{q-1}{q-1} + \cdots + \binom{3}{3}$ are $q-3$.

By Macaulay's theorem we have
\begin{center}
$h'_{q+1} \leq ((h'_{q})_{q})^{+1}_{+1}= \binom{q+2}{q+1}+ \binom{q}{q} + \cdots + \binom{4}{4}$,
\end{center}
the last terms are $q-3$.

By Macaulay's theorem we have
\begin{center}
$h'_{q+2} \leq  \binom{q+3}{q+2}+ \binom{q+1}{q+1} + \dots + \binom{5}{5}= q+3+q-3=2q$,
\end{center}

then $2q+1 \leq 2q$.

It is a contradiction. The result follows.

\end{proof}

\section{Asymptotic behavior of the minimum}
In this section we give a new proof of part of Theorem $3.6$ in \cite{MNZ2}.

Let $P_{m} = m + \binom{m+d-2}{d-1}$ be the codimension of a full Perazzo algebra of type $m$. Denote by $\mu_{d,k}(r)$ the minimal entry in degree $k$ of a Gorenstein $h$-vector with codimension $r$ and socle degree $d$.

\begin{lem}\label{lfp}
$\mu_{d,k}(P_{m}) \geq \binom{m+d-k-1}{d-k}$.
\end{lem}

\begin{proof}
We proceed by induction on $k$.

For $k=1$, we have $\mu_{d,1}(P_{m}) = P_{m} = m + \binom{m+d-2}{d-1} > \binom{m+d-2}{d-1}$.

Now, suppose that $$\mu_{d,k-1}(P_{m}) > \binom{m+d-k}{d-k+1}.$$ From Theorem $2.4$ in \cite{MNZ2}, we get $$\mu_{d,k}(P_{m}) \geq \left(\left(\mu_{d,k-1}(P_{m})\right)_{(d-k+1)}\right)_{-1}^{-1} + \left(\left(\mu_{d,k-1}(P_{m})\right)_{(d-k+1)}\right)_{-d+2k+1}^{-d+2k}.$$ By inductive hypothesis, and by basic properties of binomial expansions, we have: $\left(\left(\mu_{d,k-1}(P_{m})\right)_{(d-k+1)}\right)_{-1}^{-1} > \binom{m+d-k-1}{d-k}$ and $\left(\left(\mu_{d,k-1}(P_{m})\right)_{(d-k+1)}\right)_{-d+2k+1}^{-d+2k} > \binom{m+k}{k+2}.$

So, $$\mu_{d,k}(P_{m}) \geq \binom{m+d-k-1}{d-k} + \binom{m+k}{k+2} > \binom{m+d-k-1}{d-k}.$$ as we wanted.

\end{proof}

\begin{thm}[Migliore-Nagel-Zanello \cite{MNZ2}] \label{thm:MNZ}
Let $A$ be a Gorenstein algebra of codimension $r$ and socle degree $d$. Then, for all $k < \lfloor d/2 \rfloor$ $$\lim_{r \rightarrow \infty}\dfrac{\mu_{d,k}(r)}{r^{\frac{d-k}{d-1}}} = \dfrac{((d-1)!)^{\frac{d-k}{d-1}}}{(d-k)!}.$$  
\end{thm}

\begin{proof} For any integer $r >> 0$ there is a unique integer $P_{m} = m + \binom{m+d-2}{d-1}$ such that $$P_{m} \leq r \leq P_{m+1}.$$ Applying the function $\mu_{d,k}$ we have \[\mu_{d,k}(P_{m}) \leq \mu_{d,k}(r) \leq \mu_{d,k}(P_{m+1}).\] By Lemma \ref{lfp}   \[\binom{m+d-k-1}{d-k} \leq \mu_{d,k}(r) \leq \binom{m+d-k}{d-k} + \binom{m+k}{k}.\] Therefore $$\frac{m^{d-k}}{(d-k)!} + o(m^{d-k-1}) \leq \mu_{d,k}(r) \leq \frac{m^{d-k}}{(d-k)!} + o(m^{d-k-1})$$ where $o(m^{s})$ denote all terms of degree less than $s$.

On other hand, since $P_{m} \leq r \leq P_{m+1}$, then
\begin{eqnarray*}
\frac{m^{d-1}}{(d-1)!} + o(m^{d-2})  \leq & r & \leq  \frac{m^{d-1}}{(d-1)!} + o(m^{d-2}) \\
\frac{m^{d-k}}{((d-1)!)^{\frac{d-k}{d-1}}} + o(m^{d-k-1})  \leq & r^{\frac{d-k}{d-1}} & \leq \frac{m^{d-k}}{((d-1)!)^{\frac{d-k}{d-1}}} + o(m^{d-k-1}) \\
\dfrac{1}{\frac{m^{d-k}}{((d-1)!)^{\frac{d-k}{d-1}}} + o(m^{d-k-1})} \leq & \dfrac{1}{r^{\frac{d-k}{d-1}}} & \leq \dfrac{1}{\frac{m^{d-k}}{((d-1)!)^{\frac{d-k}{d-1}}} + o(m^{d-k-1})}
\end{eqnarray*}
Multiplying, we get  $$\dfrac{\frac{m^{d-k}}{(d-k)!} + o(m^{d-k-1})}{\frac{m^{d-k}}{((d-1)!)^{\frac{d-k}{d-1}}} + o(m^{d-k-1})} \leq \dfrac{\mu_{d,k}(r)}{r^{\frac{d-k}{d-1}}} \leq \dfrac{\frac{m^{d-k}}{(d-k)!} + o(m^{d-k-1})}{\frac{m^{d-k}}{((d-1)!)^{\frac{d-k}{d-1}}} + o(m^{d-k-1})}.$$ Since in both sides the limit exists and are the same, the result follows.
    
\end{proof}

{\bf Acknowledgments}. We would like to thank the anonymous referee for the careful reading, suggestions and corrections of a previous version of the manuscript. \\
A CNPq Research Fellowship partially supported the second named author (Proc. 309094/2020-8). The first author was financed in part by the Cordena\c{c}\~{a}o de Aperfei\c{c}oamento de Pessoal de N\'{i}vel Superior - Brasil (CAPES) - Finance Code 001.

\end{document}